%% file: critical.tex
\documentclass[10pt]{amsart}
\input{header}

\title{Exotic diffeomorphisms on $4$-manifolds with $b_2^+ = 2$}
\author{Haochen Qiu}

\begin{document}

\maketitle

\begin{abstract}
    While the exotic diffeomorphisms turned out to be very rich, we know much less about the $b^+_2 =2$ case, as parameterized gauge-theoretic invariants are not well defined. In this paper we present a method (that is, comparing the winding number of parameter families) to find exotic diffeomorphisms on simply-connected smooth closed $4$-manifolds with $b^+_2 =2$, and as a result we obtain that $2\CP^2 \# 10 \overline{\CP^2}$ admits exotic diffeomorphisms. This is currently the smallest known example of a closed $4$-manifold that supports exotic diffeomorphisms.
\end{abstract}

\tableofcontents
%\charpter{Introduction}
\section{Introduction}
The first examples of exotic diffeomorphisms on simply-connected smooth closed $4$-manifolds were found by Ruberman\cite{Ruberman1998AnOT} using parameterized Donaldson invariant, and his examples have $b_2^+ \ge 4$. While the exotic diffeomorphisms turned out to be very rich, we know much less about the $b^+_2 =2$ case, because parameterized gauge-theoretic invariants are not well defined. In this paper we present a method to find exotic diffeomorphisms on simply-connected smooth closed $4$-manifolds with $b^+_2 =2$, and as a result we obtain
\begin{theorem}
$2\CP^2 \# 10 \overline{\CP^2}$ admits exotic diffeomorphisms.
\end{theorem}
To motivate the method, we first discuss the complications due to small $b^+_2$ in the ordinary case.
%We introduce the difficulty of the small $b^+_2$ from the ordinary case.

The ordinary Seiberg-Witten invariant depends on the choice of the metric and the perturbing self-dual $2$-form. All of such parameters can be separated to some ``chambers''. The Seiberg-Witten invariant remains constant for parameters within the same chamber. Furthermore, it remains unchanged for a parameter and its pushforward by a diffeomorphism. The space of parameters is equivalent to $S^{b^+_2-1}$, hence when $b^+_2 >1$, there is only one chamber and the ordinary Seiberg-Witten invariant is a well-defined smooth invariant. When $b^+_2 =1$, there are two chambers. Szab\'o's result\cite{Szabo96} says there are two homeomorphic smooth $4$-manifolds with $b^+_2 =1$, such that the Seiberg-Witten invariant for one of them is some $m$ for one chamber, and the invariant for another one can not be $m$ for any chamber. This proves that they are not diffeomorphic. Such $4$-manifolds are the smallest ones (in the sense of $b^+_2$) that admit exotic smooth structures detected by the gauge theory.

To detect exotic diffeomorphisms, we compute the family Seiberg-Witten invariants (FSW) for the mapping tori of two diffeomorphisms, and if they are different, these tori are not diffeomorphic, hence these diffeomorphisms are not smoothly isotopic. By this machinery Ruberman and Baraglia-Konno prove that for $X$ with an exotic smooth structure detected by the Seiberg-Witten invariant and $b^+_2(X) >1$, $X\# \CP^2 \#2 \overline{\CP^2}$ and $X\# S^2 \times S^2$ admit exotic diffeomorphisms. Note that these manifolds have $b^+_2 >2$. 

Our work generalize such results to $b^+_2 =2$. The main issue in this case is that, the family invariant $FSW$ on a family of manifolds, depends on the family of parameters. The mapping torus is an $S^1$-family of manifolds, so the space of parameter families is an $S^1$-family of $S^{b^+_2-1}=S^1$. The set of chambers corresponds to the set of fiberwise homotopy classes of these parameter families, which has more than one elements. If there exists a bundle isomorphism between two mapping tori, it would bring a chamber on one mapping torus to a chamber on another one. To disprove this hypothesis, we need to compare the $FSW$ for these chambers. But the situation is a bit more complicated than in the ordinary case treated by Szab\'o: 

%we don't have the good situation of Szab\'o:

\begin{enumerate}
\item[\textbullet] $FSW$ may run over all possible values ($\Z$ if it is an integer invariant, or $\Z/2$ for the mod $2$ invariant) as the chambers change.
\item[\textbullet] The set of chambers corresponds to $\Z$ or $\Z/2$ only noncanonically, which means we can only measure the difference between two parameter families on the same mapping torus. But we cannot compare two chambers on different mapping tori.
\item[\textbullet] The family of metrics will also determine the chamber, but we don't know how the diffeomorphism in the hypothesis acts on the families of metrics.
\end{enumerate}
To solve all these problems, we construct a homotopy invariant of the parameter families, which is called the winding number. We prove that this is an invariant under the diffeomorphism of mapping tori. This viewpoint symplifies the chamber structure and decouples the families of metrics and the family of perturbing $2$-forms. By additional assumption on $b^-_2$ we can throw out the influences of the metric family and the $\text{Spin}^c$-structure, such that we can apply the traditional wall-crossing and gluing arguments.

As a remark, Konno-Mallick-Taniguchi\cite{konno2024exoticdehntwists4manifolds} noted that $2\CP^2 \# n \overline{\CP^2}$ admits exotic diffeomorphisms for $n> 10$ by a fairly different method. As far as we know, our result is the smallest example of a closed $4$-manifold that supports exotic diffeomorphisms. Furthermore, our result is given by  the families Seiberg-Witten invariant so we are able to detect an exotic diffeomorphism of infinite order in the mapping class group.

\subsection*{Acknowledgements}
The author wants to thank his advisor Daniel Ruberman for asking the main problem of this work and for his suggestions on the references. The author wants to thank Hokuto Konno for valuable comments and discussions. This work is partially supported by NSF grant DMS-1952790.

\subsection{Prerequisite}

%\begin{lemma}[Gluing formula for $b^+ = 1$]
%Let $\mu_M$ be a perturbation on $M$. Let $\mu_N$ be a generic section of $H^+(N)\times \S^1 \to \S^1$. 
%\end{lemma}

Let $H^2(X;\R)$ be the vector space of harmonic $2$-forms. The cup product gives a nondegenerate quadratic form on $H^2(X;\R)$ with positive index $b_2^+$. All the positive definite rank $b_2^+$ subspaces in $H^2(X;\R)$ form an open subset of the rank $b_2^+$ Grassmanian in $H^2(X;\R)$. It will be denoted by $G_{b_2^+}^+(H^2(X;\R))$. Every metric $g$ on $X$ gives a unique Hodge star operator $\star_g$ by the formula 
\[
\int \alpha \wedge \star_g \beta = \langle \alpha, \beta \rangle_g.
\]
We say a $2$-form $\alpha$ is self-dual if $\star_g \alpha = \alpha$. From the above formula we see that all self-dual harmonic $2$-forms under $\star_g$ form a positive definite rank $b_2^+$ subspace $H^+_g$. For any $2$-form $\alpha$ we define the self dual part $\alpha^+$ of $\alpha$ by $\alpha^+ = (\alpha +  \star_g\alpha)/2$. Although the definition of $\alpha^+$ depends on the choice of a metric, in this paper we will omit it and $\alpha^+$ should be understood as a self dual part of a $2$-form under a suitable metric (or a family of $2$-forms under a family of metrics).

For later usage we state the following proposition from the homotopy theory.
\begin{prop}\label{prop:fiber-homotopy-class}
Let $B$ be a circle and $\mathcal{P}$ be an $\S^1$-bundle over $B$. The fiberwise homotopy class of sections on $\mathcal{P}$, denoted by $[B,\mathcal{P}]_f$, can be identified with $\Z$ or $\Z/2$ noncanonically. It becomes canonical after choosing a trivialization of $\mathcal{P}$. 
\end{prop}
\begin{proof}
For two sections of $\mathcal{P}$, the obstruction of the fiberwise homotopy is measured by a sequence of elements which live in $H^r(B;\pi_r^{loc}(\mathcal{P}))$, where $\pi_r^{loc}(\mathcal{P})$ is the local system of $r$-th homotopy groups of the fiber (\cite{Steenrod51}).  

Since the fiber of $\mathcal{P}$ is just $\S^1$, we only need to consider $r=1$. If $\mathcal{P}$ is orientable, $\pi_1^{loc}(\mathcal{P})$ is the constant sheaf $\underline \Z$ and $H^1(B;\pi_1^{loc}(\mathcal{P})) = \Z$. If $\mathcal{P}$ is a Klein bottle, $\pi_1^{loc}(\mathcal{P})$ is the twisted coefficient and $H^1(B;\pi_1^{loc}(\mathcal{P})) = \Z/2$ (see \cite{BT82} Exercise 10.7).
\end{proof}

\section{Winding number of parameter families}

\begin{definition}\label{def:S-g}
Let $B=\S^1$ be the parameter space and $X$ be a smooth closed oriented $4$-manifold with $b_2^+ = 2$. For any metric family $g$ indexed by $B$, the family of nonzero self-dual harmonic $2$-forms is homotopic to an $\S^1$-bundle
\[
\mathcal{S}_{g} := \bigsqcup_{b\in B} (H^+_{g(b)} - \{0\})
\]
over $B$. 
\end{definition}

Since in the construction that two diffeomorphisms have different family invariants, we cannot compare two chambers on the same manifold (see \cite{BK2020} Theorem 9.7 for the analogue in large $b^+_2$ case), we do need a canonical trivialization of the bundle $\mathcal{S}_{g}$ to identify ``absolute positions'' for all chambers: Choose a base point $P$ in $G_{b_2^+}^+(H^2(X;\R))$, then project every positive planes to $P$. This projection is nondegenerate. This gives a trivialization of $\mathcal{S}_{g}$.

%The integral of the pushforward of the volumn form on each plane gives a smooth function on $G_{b_2^+}^+(H^2(X;\R))$ with a unique maximum at $P$. The upward gradient flow associated to this function gives a homotopy that shrinks $G_{b_2^+}^+(H^2(X;\R))$ to a point. 

A parameter family $(g,\mu)$ is regular if no reducible solutions occur. Let $\mathcal{H}$ be the harmonic projection of $2$-forms. Then a parameter family $(g,\mu)$ is regular if and only if $\mathcal{H}(\mu(b)) \neq 0$ for any $b\in B$.

\begin{definition}\label{def:winding-number}
Adopt the settings in Definition \ref{def:S-g}. If the family $E_X$ over $B$ has structure group that preserves the orientation of the positive cone of $H^{2}(X,\R)$, then the \textbf{winding number} $wind$ of a parameter family $(g,\mu)$ is defined by the degree of the projection of $\mathcal{H}(\mu)$ to the fiber of $\mathcal{S}_{g }$ (note that the trivialization of $\mathcal{S}_{g}$ is canonical). If $\mathcal{S}_{g}$ is nonorientable, the winding number of $(g,\mu)$ can be defined similarly, but with values in $\Z/2$.
\end{definition}

\begin{example}\label{example:perturbation-family-path-Z2}
If $X = Y \# (\S^2\times \S^2)$ with $b^+_2 (Y) =1$, and $E_X$ is the mapping torus of $id_Y\# r$ where $r$ sends $(a,b)\in H^2(\S^2\times \S^2)$ to $(-a,-b)$. Then there are two homotopy classes of purterbation families. They are represented by two paths in Figure \ref{fig:chamber-1}.
\begin{figure}[ht!]
    \begin{Overpic}{\includegraphics[scale=0.6]{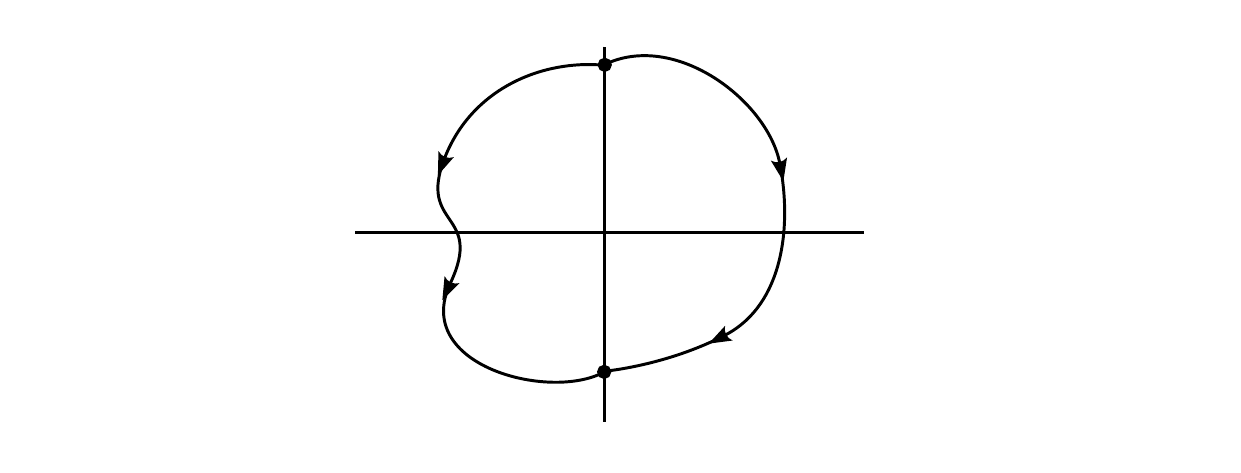}}
        \put(22,21){$H^{2,+}(Y)$}
        \put(44,0){$H^{2,+}(\S^2\times \S^2)$}
    \end{Overpic}
    \caption{}
    \label{fig:chamber-1}
\end{figure}
\end{example}

\subsection{Properties of the winding number}
To disprove the hypothesis that there exists a diffeomorphism of mapping tori, we have to first figure out how such diffeomorphisms can change the parameter family. It turns out that the winding numbeer is an invariant under the diffeomorphism of any $\S^1$-parameter family of positve index $2$ manifolds. 
\begin{lemma}
Let $B=\S^1$ be the parameter space and $X$ be a smooth closed oriented $4$-manifold with $b_2^+ = 2$. Let $E_X$ be a smooth family of $X$ over $B$. Let $(g,\mu)$ be a parameter family that doesn't cross any wall. Then for any diffeomorphism $f$ of $E_X$, 
\[
wind(g,\mu) = wind(f^*g,f^*\mu).
\]
\end{lemma}
It turns out that when $b^-_2$ is not so large, we don't really need to know how a diffeomorphism acts on the set of parameter families. Hence we omit the proof of the above lemma.

For each $g$, the harmonic projection gives an identification between the period bundle in \cite{LL01} Section 3.1 and the bundle $\mathcal{S}_g$, so the winding number actually determines the chamber:%$(g_1, 2\pi K + \mu_1)$ and $(g_2, 2\pi K + \mu_2)$ are in the same chamber. 

\begin{lemma}\label{lem:szabo-lemma-3.1}
Let $X$ be a smooth closed oriented $4$-manifold with $H_1(X,\Z)=0$ and $b_2^+(X) = 2$. Fix a homology orientation of $H^{2,+}(X, \R)$. Let $E_X$ be a family of $X$ over the circle $B$. For each Riemannian metric family $g$ let the bundle $\mathcal{S}_{g} $ be the one defined in Definition \ref{def:S-g} with the orientation compatible with the homology orientation. Then for each characteristic element $K\in H^2(X;\Z)$ with $(K^2 -2e(X) -3sign(X))/4 \ge 0$ we have
\begin{enumerate}
\item[\textbullet] If the winding number of $(g_1, 2\pi K^+ + \mu_1)$ is equal to the winding number of $(g_2, 2\pi K^+ + \mu_2)$, then 
\[
FSW(E_X, K, g_1 ,\mu_1 ) =FSW(E_X, K, g_2 ,\mu_2 )
\] 
\end{enumerate}
\end{lemma}
\begin{proof}
It suffices to find a homotopy from $(g_1,\mu_1)$ to $(g_2,\mu_2)$ without crossing any wall. Since the space of all metrics is contractible, there exsists a homotopy $g(b,t)$ between the loops $g_1$ and $g_2$.  To avoid $2\pi K^+ + \mathcal{H}(\mu_1)$ being zero during this process, we may change the perturbation family by the following procedure:

For each $g(b,t)$, the projection of the base plane $P$ to $H^+_{g(b,t)}$ is nondegenerate. Hence a frame in $P$ will induce a smooth family of frames on $H^+_{g}$. For each $g(b,t)$, the pairing between the attached frame on $H^+_{g(b,t)}$ and an element in $H^+_{g(b,t)}$ gives an element in $\R^2$ (this identification is compatible with the previous trivialization of $\mathcal{S}_g$). Denote this identification by 
\[
i_{b,t}:H^+_{g(b,t)} \to \R^2.
\]
Let
\[
\mu(b,t) := \mu_1(b) + 2\pi(i_{b,t}^{-1} i_{b,0}K^{+_{g(b,0)}} -K^{+_{g(b,t)}})
\]
where $K^{+_{g(b,t)}}$ is the self dual part of $K$ with respect to the netric $g(b,t)$. Then $(g(b,t),\mu(b,t))$ gives a homotopy from $(g_1,\mu_1)$ to $(g_2,\mu_1')$, and $2\pi K^{+_g} + \mathcal{H}(\mu)$ is always nonzero. Moreover this homotopy preserves the winding number since $i_{b,t}\mu(b,t) $ is actually unchanged. This means that
\[
wind(g_2, 2\pi K^+ + \mu_1') = wind(g_1, 2\pi K^+ + \mu_1) =wind(g_2, 2\pi K^+ + \mu_2).
\]
Then by Proposition \ref{prop:fiber-homotopy-class}, $(g_2,\mu_1')$ is fiberwise homotopy equivalent to $(g_2,\mu_2)$. The concatenation of these two homotopies is a homotopy from $(g_1,\mu_1)$ to $(g_2,\mu_2)$ without crossing any wall. This proves that 
\[
FSW(E_X, K, g_1 ,\mu_1 ) =FSW(E_X, K, g_2 ,\mu_2 ).
\]

\end{proof}

\section{Proof of the main theorem}
If the characteristic element $K\in H^2(X,\Z)$ is not in the positive cone of $H^2(X,\Z)$, there is a risk that the metric family $g$ indexed over a circle would make $K^+$ wind around the origin of $H^+_g$. This means that the chamber of the parameter family may depend on  $g$, which is something we cannot control in the diffeomorphism. We have to assume some conditions on $b_2^-(X)$ to avoid this situation: %If $b_2^-(X)$ is large enough, every basic class will live in the 
\begin{lemma}\label{lem:szabo-lemma-3.2}
Let $X$ be a smooth closed oriented $4$-manifold with $H_1(X,\Z)=0$, $b_2^+(X) = 2$ and $b_2^-(X) \leq 10$. Let $E_X$ be a family of $X$ over a circle $B$. Then for every characteristic element $K\in H^2(X,\Z)$, family of metrics $g_1$, $g_2$, and sufficiently small (with respect to $K^+$) perturbation family $\mu_1$, $\mu_2$, %such that $(g_i, \mu_i)$ has the same winding number, 
we have 
\[
FSW(E_X, K, g_1 ,\mu_1 ) =FSW(E_X, K, g_2 ,\mu_2 ).
\] 
\end{lemma}
\begin{proof}
Let $K \in H^2(X;\Z)$ be a characteristic element for which the formal dimension of the parameterized moduli space is at least $0$. This means that the formal dimension of the ordinary moduli space is $-1$. Then 
\[
2e(X) + 3\text{sign}(X) = 4 +5b_2^+(X) - b^-_2(X) \geq 4
\]
implies $K^2 \geq 0$. Hence %$2\pi K \in G_{b_2^+}^+(H^2(X;\R))$, and therefore 
the distance between $2\pi K$ and the subspace perpendicular to $H^+_g \subset H^2(X;\R)$ has a lower bound. As a corollary we have that for small enough $\mu_1$ and $\mu_2$,
\[
wind(g_1, 2\pi K^+ -\mu_1) = wind(g_2, 2\pi K^+ -\mu_2).
\]
Now Lemma \ref{lem:szabo-lemma-3.2} follows from Lemma \ref{lem:szabo-lemma-3.1}.
\end{proof}

Now we can generalize the exotic diffeomorphims found in \cite{BK2020} Theorem 9.7 to our case:
\begin{theorem}\label{BK-Thm-9.7}
Let $M$ and $M'$ be two closed oriented simply connected smooth $4$-manifolds with indefinite intersection forms. Suppose that $\phi: M\to M'$ is a homeomorphism. Suppose that $b_2^+(M) = 1, b_2^-(M) \leq 9$ and that $\ss_M$ is a $\text{Spin}^c$ -structure on $M$ with $d(M,\ss_M) = 0$. We make the following assumptions:
\begin{enumerate}
\item[1)] Suppose that for any matric $g_M$ and sufficiently small (with respect to $c_1(\ss_M)^+$) perturbation $\mu_M$, all solutions on $M$ are irreducible and $SW(M, \ss_M, g_M, \mu_M)$ has the same value. Similarly, suppose that for any matric $g_{M'}$ and sufficiently small perturbation $\mu_{M'}$, all solutions on $M'$ are irreducible and $SW({M'}, \phi(\ss_M), g_{M'}, \mu_{M'})$ is well defined.; 
\item[2)]  Suppose that 
\begin{equation}\label{equ:assumption-sw-nonequal}
SW(M, \ss_M, g_M, \mu_M) \neq SW({M'}, \phi(\ss_M), g_{M'}, \mu_{M'}) \mod 2;
\end{equation}
\item[3)] Lastly, suppose that $M\# (\S^2\times \S^2)$ is diffeomorphic to $M'\# (\S^2\times \S^2)$.
\end{enumerate}
Then $M\# (\S^2\times \S^2)$ admits an exotic diffeomorphism. 
\end{theorem}
\begin{proof}
Let $\ss_0$ be the $\text{Spin}^c$ -structure on $\S^2\times \S^2$ which has trivial determinant line. Let $r$ be the reflection of $\S^2\times \S^2$ in Example \ref{example:perturbation-family-path-Z2}.

Let $\phi' : M\# (\S^2\times \S^2) \to M'\# (\S^2\times \S^2)$ be a homeomorphism induced by $\phi$ and $id_{\S^2\times \S^2}$. Since $M$ is indeinite, there exsists a diffeomorphism $\psi : M\# (\S^2\times \S^2) \to M'\# (\S^2\times \S^2)$ such that $\psi$ and $\phi'$ induce the same map on cohomology groups (see \cite{Wall64}). Then we have
\[
\psi(\ss_M\# \ss_0) = \phi(\ss_M)\#\ss_0.
\]

Define \[
f_1 := id_M\# r: M\# (\S^2\times \S^2)\to  M\# (\S^2\times \S^2)
\] 
and 
\[
f_2 := \psi^{-1} \circ (id_{M'} \# r ) \circ \psi:M\# (\S^2\times \S^2)\to  M\# (\S^2\times \S^2).
\]
We will show that they are continuously isotopic but not smoothly isotopic. Then $f_1\circ f_2^{-1}$ will be an exotic diffeomorphism.

Since $r$ is a reflection, it induces the identity on $H^2(\S^2\times \S^2;\Z/2)$, and therefore so does $id_{M'} \# r $ on the $H^2(\S^2\times \S^2;\Z/2)$-summand of $H^2(M'\# (\S^2\times \S^2);\Z/2)$. On the other hand $\psi^* = (\phi')^*$ acts on the $H^2(M';\Z/2)$ summand of $H^2(M'\# (\S^2\times \S^2);\Z/2)$. Hence $f_1$ and $f_2$ induce the same map on $H^2(M\# (\S^2\times \S^2);\Z/2)$. Therefore by \cite{Quinn86}, $f_1$ and $f_2$ are continuously isotopic. 

Define \[
f_1' := id_{M'}\# r: M'\# (\S^2\times \S^2)\to  M'\# (\S^2\times \S^2).
\] 
Let $E_1$, $E_1'$ and $E_2$ be the mapping tori of $f_1$, $f_1'$ and $f_2$. Choose $g_2$, $\mu_2$ to be any generic family of parameters on $E_2$. Note that $\psi$ gives a bundle isomorphism between $E_1'$ and $E_2$. Hence 
\begin{align}\label{equ:bundle-iso}
\text{FSW}^{\Z/2} (E_2, \ss_M \# \ss_0, g_2, \mu_2) &= 
 \text{FSW}^{\Z/2} (E_1', \psi(\ss_M \# \ss_0), \psi(g_2), \psi(\mu_2)) \\
 &=  \text{FSW}^{\Z/2} (E_1', \phi(\ss_M) \# \ss_0, \psi(g_2), \psi(\mu_2)).\notag
\end{align}

Choose $g_1$, $\mu_1$ to be any generic family of parameters on $E_1$ with $\mu_1$ small enough (with respect to $c_1(\ss_M)^+$). By the gluing construction \cite{BK2020}, the parameterized moduli space of $E_1$ comes from gluing an irreducible solution on $M$ and a reducible solution on $\S^2\times \S^2$, which locates on a parameter where the self-dual $2$-form on $\S^2\times \S^2$ is $0$ (see Figure \ref{fig:wall-crossing}). Similarly, choose $g_2$, $\mu_2$ to be any generic family of parameters on $E_2$ with $\mu_2$ small enough (with respect to $c_1(\ss_{M})^+$).
\begin{figure}[ht!]
    \begin{Overpic}{\includegraphics[scale=0.6]{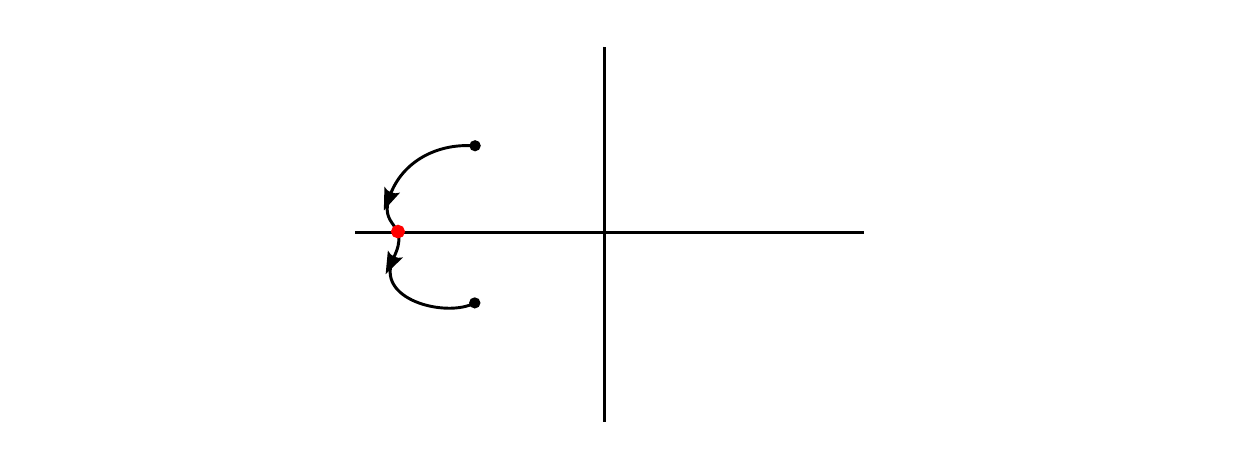}}
        \put(17,18){$H^{2,+}(M)$}
        \put(44,0){$H^{2,+}(\S^2\times \S^2)$}
    \end{Overpic}
    \caption{The only solutions on $\S^2\times \S^2$ are reducible by the dimension reason and therefore locate on the point with $0$ perturbation on $H^{2,+}(\S^2\times \S^2)$, and after the gluing, the parameterized moduli space of $E_1$ locates on nearby parameters.}
    \label{fig:wall-crossing}
\end{figure}

Hence we have 
\begin{align*}
\text{FSW}^{\Z/2} (E_1, \ss_M \# \ss_0, g_1, \mu_1) &= SW(M, \ss_M, g_1(b)|_M, \mu_1(b)|_M) \mod 2 \\
\text{FSW}^{\Z/2} (E_1', \phi(\ss_M) \# \ss_0, \psi(g_2), \psi(\mu_2)) &= SW(M', \phi(\ss_M) , \psi(g_2)(b')|_{M'}, \psi(\mu_2)(b')|_{M'}) \mod 2
\end{align*}
for some $b$ in the base of $E_1$, and some $b'$ on the base of $E_1'$ by the gluing results. Since $\mu_1$, $\mu_2$ are much smaller than $c_1(\ss_M)^+$, we have $\mu_1(b)|_M << c_1(\ss_M)^+$ and $\psi(\mu_2)(b')|_{M'} << \phi(c_1(\ss_M))^+ = c_1(\phi(\ss_M))^+$. Therefore, by Assumption (1) and (2) on the ordinary Seiberg-Witten invariant, we have 
\[
\text{FSW}^{\Z/2} (E_1, \ss_M \# \ss_0, g_1, \mu_1)  \neq \text{FSW}^{\Z/2} (E_1', \phi(\ss_M) \# \ss_0, \psi(g_2), \psi(\mu_2)).
\]
Combine this and (\ref{equ:bundle-iso}) we have 
\[
\text{FSW}^{\Z/2} (E_1, \ss_M \# \ss_0, g_1, \mu_1)  \neq \text{FSW}^{\Z/2} (E_2, \ss_M \# \ss_0, g_2, \mu_2).
\]
Suppose that $f_1$ is smoothly isotopic to $f_2$, then the isotopy forms a bundle isomorphism $H$ from $E_1$ to $E_2$. Since $f_1$ and $f_2$ act by the identity on $H^2(M)$ and $\ss_0$ is a $\text{Spin}^c$-structure with trivial determinant line on $\S^2\times \S^2$, $f_1$ and $f_2$ and therefore $H$ preserve $\ss_M\# \ss_0$. Hence 
\[
\text{FSW}^{\Z/2} (E_2, \ss_M \# \ss_0, H(g_1), H(\mu_1))  \neq \text{FSW}^{\Z/2} (E_2, \ss_M \# \ss_0, g_2, \mu_2).
\]
Also, $H$ would bring $(g_1,\mu_1)$ to a generic family of parameters on $E_2$ with $\mu_2$ small enough (with respect to $c_1(\ss_M)^+$). Hence the above equation contradicts Lemma \ref{lem:szabo-lemma-3.2}.\end{proof}

\begin{corollary}
$2\CP^2 \# 10 \overline{\CP^2}$ admits exotic diffeomorphisms.
\end{corollary}
\begin{proof}
\cite{Szabo96} proves that there exists a family of manifolds $E(1)_{K_k}$ (where $K_k$ is the $k$-twist knot, as mentioned in the last part of \cite{fintushel1997knotslinks4manifolds}) that are homeomorphic to $E(1)$. Choose $M = E(1)_{K_m}$ and $M' = E(1)_{K_n}$ for $m,n \ge 0$, $m$ odd and $n$ even. We have that $b_2^+(M) = 1, b_2^-(M) \leq 9$. By \cite{Szabo96} Lemma 3.2, Assumption (1) in Theorem \ref{BK-Thm-9.7} is satisfied. Choose $\ss_M$ such that $c_1(\ss_M)$ is the Poincar\'{e} dual of a regular elliptic fiber. By \cite{Szabo96} Theorem 3.3, $SW(M, \ss_M, g_M, \mu_M)=m$ and $SW({M'}, \phi(\ss_M), g_{M'}, \mu_{M'})$ can be only $n$ or $0$. Hence Assumption (2) in Theorem \ref{BK-Thm-9.7} is satisfied. By \cite{Akbulut02}, \cite{Auckly03} and \cite{baykur2018dissolvingknotsurgered4manifolds}, both $M\# (\S^2\times \S^2)$ and $M'\# (\S^2\times \S^2)$ dissolve. Hence Assumption (3) in Theorem \ref{BK-Thm-9.7} is satisfied.
\end{proof}

\section{The order of an exotic diffeomorphism on manifolds with $b^+_2 = 2$}
In the previous section we show that there exists an exotic diffeomorphism $f_1\circ f_2^{-1}$. In this section we determine the order of it by the integral families invariant. Note that since $f_1\circ f_2^{-1}$ preservs the homology orientation, the families Seiberg-Witten invariant is a well-defined $\Z$-valued invariant for a given parameter family.

\begin{theorem}
Use the notations and assumptions of Theorem \ref{BK-Thm-9.7}. Let $E({f}^n)$ be the mapping torus of the $n$-fold composition of $f :=f_1\circ f_2^{-1}$. Then for any family of metrics $g$ and sufficiently small (with respect to $c_1(\ss_M\# \ss_0)^+$) perturbation family $\mu$
\[
\text{FSW}^{\Z} (E({f}^n), \ss_M \# \ss_0, g, \mu)\neq 0.
\]
\end{theorem}
\begin{proof}
We choose $g_1$, $\mu_1$ to be any generic family of parameters on $E_1$ with $\mu_1$ small enough (with respect to $c_1(\ss_M)^+$). Regard the base space $B$ as the interval $[0,1]$ with endpoints identified. Then $g_1, \mu_1$ are functions on $[0,1]$ with $f_1(g_1(1)) = g_1(0)$ and $f_1(\mu_1(1)) = \mu_1(0)$. Let $g_2$ be a path that connects $g_1(1)$ to $f_2(g_1(1))$. Let $\mu_2$ be a sufficiently small path that connects $\mu_1(1)$ to $f_2(\mu_1(1))$. Let $g$ be the concatenation of $g_1$ and $g_2$. Let $\mu$ be the concatenation of $\mu_1$ and $\mu_2$ (see Figure \ref{fig:concatenation}). 

\begin{figure}[ht!]
    \begin{Overpic}{\includegraphics[scale=0.8]{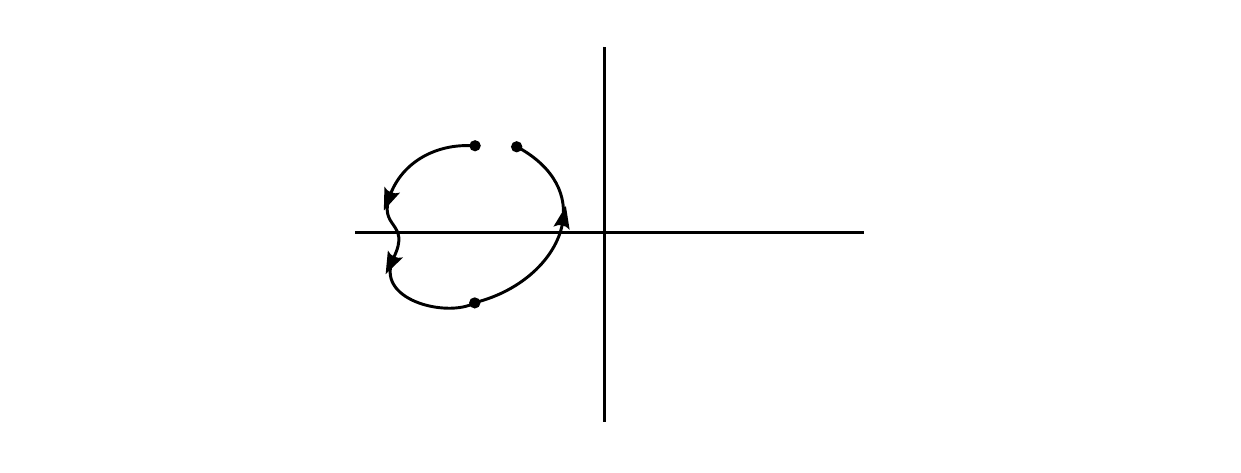}}
    \put(33,27){$\mu_1(0)$}
    \put(37,11){$\mu_1(1)$}
    \put(40,27){$f_2(\mu_1(1))$}
     \put(28,22){$\mu_1$}
      \put(42,22){$\mu_2$}
       % \put(17,18){$H^{2,+}(M)$}
       % \put(44,0){$H^{2,+}(\S^2\times \S^2)$}
    \end{Overpic}
    \caption{}
    \label{fig:concatenation}
\end{figure}
We see that $(g_2,\mu_2)$ is a parameter family for the mapping torus $E(f_2^{-1})$, and $(g,\mu)$ is a parameter family for the mapping torus $E(f)$. Since $f_2 = f_2^{-1}$ from the construction, we have $E(f_2^{-1})= E_2$. By the definition of the parameterized moduli space we have
\[
\FM(E(f), \ss_M\# \ss_0, g, \mu) = \bigsqcup_{b\in B}\M(E_1, \ss_M\# \ss_0, g_1(b), \mu_1(b)) \bigcup \bigsqcup_{b\in B}\M(E_2, \ss_M\# \ss_0, g_2(b), \mu_2(b)).
\]
We don't know how many points with signs there are in these two sets, but as in the proof of Theorem \ref{BK-Thm-9.7} we know they have different parites. Hence
\[
FSW(E(f), \ss_M\# \ss_0, g, \mu)  \neq 0.
\]
By Lemma \ref{lem:szabo-lemma-3.2}, this nontrivial result works for any family of metrics $g$ and sufficiently small (with respect to $c_1(\ss_M\# \ss_0)^+$) perturbation family $\mu$. Similarly, for $f^n$ we have
\[
FSW(E(f^n), \ss_M\# \ss_0, g^n, \mu^n) = n FSW(E(f), \ss_M\# \ss_0, g, \mu)  \neq 0,
\]
where $(g^n,\mu^n)$ is obtained from $(g,\mu)$ by translations and concatenations.
\end{proof}

\begin{corollary}
Above $f^n$ is an exotic diffeomorphism for $n\in \Z$. 
\end{corollary}
\begin{proof}
In the proof of Theorem \ref{BK-Thm-9.7} we have shown that $f$ is continuously isotopic to the identity. Hence $f^n$ is continuously isotopic to the identity. 

Let $E$ be the trivial bundle of $M \#  (\S^2\times \S^2)$ over $B$. Let $(g_0, \mu_0)$ be a generic parameter for $M \#  (\S^2\times \S^2)$. Then 
\[
FSW(E, \ss_M\# \ss_0, g_0, \mu_0) = 0
\]
since the formal dimension of the moduli space is $-1$. By Lemma \ref{lem:szabo-lemma-3.2}, $f^n$ is not smoothly isotopic to the identity.
\end{proof}

\bibliographystyle{alpha}
%\bibliographystyle{plain}
%\bibliography{./reference/braids_links}
\bibliography{./diff}
\end{document}

%% file: header.tex
\usepackage[pagebackref]{hyperref}
\usepackage{pinlabel}
\usepackage{overpic}
\usepackage{svg}
\usepackage{sseq}
\hypersetup{
    colorlinks = true,
    citecolor = [rgb]{0, 0.5, 0.00},
}

\renewcommand*{\backref}[1]{}
\renewcommand*{\backrefalt}[4]{%
    \ifcase #1 (Not cited.)%
    \or        (Cited on page~#2.)%
    \else      (Cited on pages~#2.)%
    \fi}

\usepackage{mathtools}
\usepackage{amssymb}

\usepackage[margin=1.175in]{geometry}
\usepackage{todonotes}
\usepackage{multirow}
\usepackage{longtable}
\usepackage{enumerate}
\usepackage{mathabx}
\usepackage{amsmath}
\usepackage{amsfonts}
\usepackage{amssymb}
\usepackage{float}
\usepackage{amsthm}
\usepackage{comment}
\usepackage{amscd}
\usepackage{amsxtra}
\usepackage{epsfig}
\usepackage{slashed}

\usepackage{listings}
\usepackage{color}

\definecolor{dkgreen}{rgb}{0,0.6,0}
\definecolor{gray}{rgb}{0.5,0.5,0.5}
\definecolor{mauve}{rgb}{0.58,0,0.82}

\usepackage{color}
\usepackage[all]{xy}
\usepackage{verbatim}

\usepackage{eucal}
\usepackage{mathrsfs}

\usepackage{graphicx}
\usepackage{tikz-cd}

\usepackage{url}
\usepackage{verbatim}
\usepackage{xy}
\xyoption{all}

\usepackage{amsthm}
\usepackage{tikz}
\usetikzlibrary{decorations.pathreplacing}

\usepackage{caption} 
\makeatletter
\def\@tocline#1#2#3#4#5#6#7{\relax
  \ifnum #1>\c@tocdepth % then omit
  \else
    \par \addpenalty\@secpenalty\addvspace{#2}%
    \begingroup \hyphenpenalty\@M
    \@ifempty{#4}{%
      \@tempdima\csname r@tocindent\number#1\endcsname\relax
    }{%
      \@tempdima#4\relax
    }%
    \parindent\z@ \leftskip#3\relax \advance\leftskip\@tempdima\relax
    \rightskip\@pnumwidth plus4em \parfillskip-\@pnumwidth
    #5\leavevmode\hskip-\@tempdima
      \ifcase #1
       \or\or \hskip 1em \or \hskip 2em \else \hskip 3em \fi%
      #6\nobreak\relax
    \hfill\hbox to\@pnumwidth{\@tocpagenum{#7}}\par% <---- \dotfill -> \hfill
    \nobreak
    \endgroup
  \fi}
\makeatother

\newtheorem{lemma}{Lemma}[section]

\newtheorem{corollary}[lemma]{Corollary}

\newtheorem{theorem}[lemma]{Theorem}

\newtheorem{prop}[lemma]{Proposition}

\theoremstyle{definition}

\newtheorem{definition}[lemma]{Definition}

\newtheorem{example}[lemma]{Example}

\numberwithin{equation}{section} % 在每个小节开始时重新编号方程
\newcommand{\Z}{\mathbb{Z}}

\newcommand{\M}{\mathfrak{M}}
\newcommand{\FM}{\mathcal{F}\mathfrak{M}}

\newcommand{\CP}{\mathbf{C}P}

\newcommand{\R}{\mathbb{R}}

\renewcommand{\S}{\mathbb{S}}
\renewcommand{\ss}{\mathfrak{s}}